\documentclass{amsart}
\usepackage{setspace}
\usepackage{a4}
\usepackage{amssymb,amsmath,amsthm,latexsym}
\usepackage{amsfonts}
\usepackage{graphicx}
\usepackage{textcomp}
\usepackage{cite}
\usepackage{enumerate}
\usepackage[mathscr]{euscript}
\usepackage{mathtools}
\newtheorem{theorem}{Theorem}[section]

\newtheorem{conjecture}[theorem]{Conjecture}

\newtheorem{definition}[theorem]{Definition}

\newtheorem{remark}[theorem]{Remark}
\newtheorem{question}[theorem]{Question}
\title{This is the title}
\usepackage{amssymb}
\usepackage{hyperref}
\usepackage{amsmath}
\usepackage{tikz}
\usepackage{fancyhdr}

\begin{document}
\hrule\hrule\hrule\hrule\hrule
\vspace{0.3cm}	
\begin{center}
{\bf{p-ADIC WELCH BOUNDS AND p-ADIC  ZAUNER CONJECTURE}}\\
\vspace{0.3cm}
\hrule\hrule\hrule\hrule\hrule
\vspace{0.3cm}
\textbf{K. MAHESH KRISHNA}\\
Post Doctoral Fellow \\
Statistics and Mathematics Unit\\
Indian Statistical Institute, Bangalore Centre\\
Karnataka 560 059, India\\
Email: kmaheshak@gmail.com\\

Date: \today
\end{center}

\hrule\hrule
\vspace{0.5cm}
%--------------------------------------
\textbf{Abstract}: Let $p$ be a prime. For $d\in \mathbb{N}$, let $\mathbb{Q}_p^d$  be the standard $d$-dimensional p-adic  Hilbert space.  Let $m \in \mathbb{N}$ and   $\text{Sym}^m(\mathbb{Q}_p^d)$ be  the p-adic Hilbert   space of symmetric m-tensors. We prove the following result.  Let  $\{\tau_j\}_{j=1}^n$ be   a collection in $\mathbb{Q}_p^d$
satisfying (i) $\langle \tau_j, \tau_j\rangle =1$ for all $1\leq j \leq n$ and (ii) there exists $b \in \mathbb{Q}_p$ satisfying  $	\sum_{j=1}^{n}\langle x, \tau_j\rangle \tau_j =bx$  for all $ x \in \mathbb{Q}^d_p.$ Then 
 \begin{align}\label{WELCHNONABSTRACT}
		  \max_{1\leq j,k \leq n, j \neq k}\{|n|, |\langle \tau_j, \tau_k\rangle|^{2m} \}\geq \frac{|n|^2}{\left|{d+m-1 \choose m}\right| }.
\end{align}
We call Inequality (\ref{WELCHNONABSTRACT})   as the p-adic version of Welch bounds obtained by     Welch [\textit{IEEE Transactions on  Information Theory, 1974}].  Inequality (\ref{WELCHNONABSTRACT}) differs from the non-Archimedean Welch bound obtained recently by  M. Krishna   as one can not derive one from another. We formulate p-adic  Zauner conjecture.

\textbf{Keywords}:  p-adic number field, p-adic Hilbert space, Welch bound,  Zauner conjecture. 

\textbf{Mathematics Subject Classification (2020)}: 12J25, 46S10, 47S10, 11D88.\\

\hrule

%\tableofcontents
\hrule
\section{Introduction}
In 1974  Prof. L. Welch proved the following result \cite{WELCH}.
  \begin{theorem}\cite{WELCH}\label{WELCHTHEOREM} (\textbf{Welch Bounds})
	Let $n> d$.	If	$\{\tau_j\}_{j=1}^n$  is any collection of  unit vectors in $\mathbb{C}^d$, then
	\begin{align*}
		\sum_{1\leq j,k \leq n}|\langle \tau_j, \tau_k\rangle |^{2m}=\sum_{j=1}^n\sum_{k=1}^n|\langle \tau_j, \tau_k\rangle |^{2m}\geq \frac{n^2}{{d+m-1\choose m}}, \quad \forall m \in \mathbb{N}.
	\end{align*}
	In particular,
	\begin{align*}
	\sum_{1\leq j,k \leq n}|\langle \tau_j, \tau_k\rangle |^{2}=		\sum_{j=1}^n\sum_{k=1}^n|\langle \tau_j, \tau_k\rangle |^{2}\geq \frac{n^2}{{d}}.
	\end{align*}
	Further, 
	\begin{align*}
		\text{(\textbf{Higher order Welch bounds})}	\quad		\max _{1\leq j,k \leq n, j\neq k}|\langle \tau_j, \tau_k\rangle |^{2m}\geq \frac{1}{n-1}\left[\frac{n}{{d+m-1\choose m}}-1\right], \quad \forall m \in \mathbb{N}.
	\end{align*}
	In particular,
	\begin{align*}
		\text{(\textbf{First order Welch bound})}\quad 	\max _{1\leq j,k \leq n, j\neq k}|\langle \tau_j, \tau_k\rangle |^{2}\geq\frac{n-d}{d(n-1)}.
	\end{align*}
\end{theorem}
 It is impossible to list  all   applications of Theorem \ref{WELCHTHEOREM}. A few are:   in the study of root-mean-square (RMS) absolute cross relation of unit vectors   \cite{SARWATEMEETING}, frame potential \cite{BENEDETTOFICKUS, CASAZZAFICKUSOTHERS, BODMANNHAASPOTENTIAL}, 
 correlations \cite{SARWATE},  codebooks \cite{DINGFENG}, numerical search algorithms  \cite{XIA, XIACORRECTION}, quantum measurements 
\cite{SCOTTTIGHT}, coding and communications \cite{TROPPDHILLON, STROHMERHEATH}, code division multiple access (CDMA) systems \cite{CHEBIRA1, CHEBIRA2}, wireless systems \cite{YATES}, compressed/compressive sensing \cite{TAN, VIDYASAGAR, FOUCARTRAUHUT, ELDARKUTYNIOK, BAJWACALDERBANKMIXON, TROPP, SCHNASSVANDERGHEYNST, ALLTOP},  `game of Sloanes' \cite{JASPERKINGMIXON}, equiangular tight frames \cite{SUSTIKTROPP}, equiangular lines \cite{MIXONSOLAZZO, COUTINHOGODSILSHIRAZIZHAN, FICKUSJASPERMIXON, IVERSONMIXON2022}, digital fingerprinting \cite{MIXONQUINNKIYAVASHFICKUS}  etc.

 Theorem \ref{WELCHTHEOREM}  has been improved/different proofs were  given   in \cite{CHRISTENSENDATTAKIM, DATTAWELCHLMA, WALDRONSH, WALDRON2003, DATTAHOWARD, ROSENFELD, HAIKINZAMIRGAVISH, EHLEROKOUDJOU, STROHMERHEATH}.  In 2021 M. Krishna derived continuous version of  Theorem \ref{WELCHTHEOREM} \cite{MAHESHKRISHNA}. In 2022 M. Krishna obtained Theorem \ref{WELCHTHEOREM} for Hilbert C*-modules \cite{MAHESHKRISHNA2}, Banach spaces \cite{MAHESHKRISHNA3} and non-Archimedean Hilbert spaces \cite{MAHESHKRISHNA4}.
 
In this paper we derive p-adic Welch bounds (Theorem \ref{WELCHNON2}). We formulate  p-adic Zauner conjecture (Conjecture \ref{NZ}).

\section{p-adic Welch bounds}

Let $p$ be a prime. For $d \in \mathbb{N}$, let $\mathbb{Q}_p^d$ be the standard p-adic Hilbert space equipped with the inner product 
\begin{align*}
	\langle (a_j)_{j=1}^d,(b_j)_{j=1}^d\rangle \coloneqq \sum_{j=1}^da_jb_j,  \quad \forall (a_j)_{j=1}^d,(b_j)_{j=1}^d \in \mathbb{Q}_p^d.
\end{align*}
Let $I_{\mathbb{Q}_p^d}$ be the identity operator on $\mathbb{Q}_p^d$.  Note that $\mathbb{Q}_p^d$ is not a non-Archimedean Hilbert space as it does not satisfies Equation (2) in  \cite{MAHESHKRISHNA4} (see Page 40, \cite{PEREZGARCIASCHIKHOF}). We refer  \cite{KALISCH, DIAGANABOOK, DIAGANARAMAROSON, KHRENNIKOV, ALBEVERIO} for more on p-adic Hilbert spaces.
\begin{theorem}\label{WELCHNON1}
\textbf{(First Order  p-adic Welch Bound)}	Let $p$ be a prime and $n, d \in \mathbb{N}$. If 	   $\{\tau_j\}_{j=1}^n$ is any  collection in $\mathbb{Q}_p^d$
such that  there exists $b \in \mathbb{Q}_p$ satisfying 
\begin{align*}
	\sum_{j=1}^{n}\langle x, \tau_j\rangle \tau_j =bx, \quad \forall x \in \mathbb{Q}^d_p,
\end{align*} 
then 
\begin{align*}
 \max_{1\leq j,k \leq n, j \neq k}\left \{\left| \sum_{l=1}^n\langle \tau_l,\tau_l \rangle^2 \right|, |\langle \tau_j,\tau_k\rangle|^2\right\}\geq \frac{1}{|d|}	\left|\sum_{j=1}^n\langle \tau_j, \tau_j \rangle \right|^2.	
\end{align*}
In particular, if $\langle \tau_j, \tau_j\rangle =1$ for all $1\leq j \leq n$, then 
\begin{align*}
 \text{\textbf{(First order  p-adic  Welch bound)}} \quad \max_{1\leq j,k \leq n, j \neq k}\{|n|, |\langle \tau_j, \tau_k\rangle|^2 \}\geq \frac{|n|^2}{|d|}.	
\end{align*}
\end{theorem}
\begin{proof}
	Define $S_\tau : \mathbb{Q}_p^d\ni x \mapsto \sum_{j=1}^n\langle x, \tau_j\rangle \tau_j \in \mathbb{Q}_p^d$. Then 
\begin{align*}
	&bd=\operatorname{Tra}(bI_{\mathbb{Q}_p^d})=\operatorname{Tra}(S_{\tau})=\sum_{j=1}^n\langle \tau_j, \tau_j \rangle , \\
	& b^2d=\operatorname{Tra}(b^2I_{\mathbb{Q}_p^d})=\operatorname{Tra}(S^2_{\tau})=\sum_{j=1}^n\sum_{k=1}^n\langle \tau_j, \tau_k \rangle\langle \tau_k, \tau_j \rangle.
\end{align*}	
Therefore 

\begin{align*}
	\left|\sum_{j=1}^n\langle \tau_j, \tau_j \rangle \right|^2&=|\operatorname{Tra}(S_{\tau})|^2=|bd|^2=|d||b^2d|	
	=|d|\left|\sum_{j=1}^n\sum_{k=1}^n\langle \tau_j, \tau_k \rangle\langle \tau_k, \tau_j \rangle\right|\\
	&=|d|\left| \sum_{l=1}^n\langle \tau_l,\tau_l \rangle^2+\sum_{j,k=1, j\neq k}^n\langle \tau_j,\tau_k\rangle \langle \tau_k,\tau_j\rangle\right|\\
&\leq |d| \max_{1\leq j,k \leq n, j \neq k}\left \{\left| \sum_{l=1}^n\langle \tau_l,\tau_l \rangle^2 \right|, |\langle \tau_j,\tau_k\rangle \langle \tau_k,\tau_j\rangle| \right\}\\
&= |d| \max_{1\leq j,k \leq n, j \neq k}\left \{\left| \sum_{l=1}^n\langle \tau_l,\tau_l\rangle^2 \right|, |\langle \tau_j,\tau_k\rangle|^2\right\}.
\end{align*}
Whenever $\langle \tau_j, \tau_j\rangle =1$ for all $1\leq j \leq n$, 
\begin{align*}
	|n|^2\leq |d|\max_{1\leq j,k \leq n, j \neq k}\{|n|, |\langle \tau_j, \tau_k\rangle|^2 \}.
\end{align*}
\end{proof}
We next derive higher order p-adic  Welch bounds. For this, we need the following result.
  \begin{theorem}\cite{COMON, BOCCI}\label{SYMMETRICTENSORDIMENSION}
  	If $\mathcal{V}$ is a vector space of dimension $d$ and $\text{Sym}^m(\mathcal{V})$ denotes the vector space of symmetric m-tensors, then 
  	\begin{align*}
  		\text{dim}(\text{Sym}^m(\mathcal{V}))={d+m-1 \choose m}, \quad \forall m \in \mathbb{N}.
  	\end{align*}
  \end{theorem}
  
  \begin{theorem}\label{WELCHNON2}
(\textbf{Higher Order p-adic  Welch Bounds}) Let $p$ be a prime and $n, d, m \in \mathbb{N}$.   
If 	   $\{\tau_j\}_{j=1}^n$ is any  collection in $\mathbb{Q}_p^d$
such that  there exists $b \in \mathbb{Q}_p$ satisfying 
\begin{align*}
	\sum_{j=1}^{n}\langle x, \tau_j^{\otimes m}\rangle \tau_j^{\otimes m} =bx, \quad \forall x \in \text{Sym}^m(\mathbb{Q}_p^d),
\end{align*} 
then 
\begin{align*}
\max_{1\leq j,k \leq n, j \neq k}\left \{\left| \sum_{l=1}^n\langle \tau_l,\tau_l\rangle^{2m} \right|, |\langle \tau_j,\tau_k\rangle|^{2m}\right\}\geq \frac{1}{\left|{d+m-1 \choose m}\right|}\left|\sum_{j=1}^n\langle \tau_j, \tau_j \rangle^m \right|^2.	
\end{align*}
In particular, if $\langle \tau_j, \tau_j\rangle =1$ for all $1\leq j \leq n$, then
\begin{align*}
 \text{\textbf{(Higher order p-adic  Welch bound)}} \quad  \max_{1\leq j,k \leq n, j \neq k}\{|n|, |\langle \tau_j, \tau_k\rangle|^{2m} \}\geq \frac{|n|^2}{\left|{d+m-1 \choose m}\right| }.	
\end{align*}
 \end{theorem}
  \begin{proof}
 Define $S_\tau : \text{Sym}^m(\mathbb{Q}_p^d)\ni x \mapsto \sum_{j=1}^n\langle x, \tau_j^{\otimes m}\rangle \tau_j^{\otimes m} \in \text{Sym}^m(\mathbb{Q}_p^d)$. Then 
    Then 
    \begin{align*}
    	&b\operatorname{dim(\text{Sym}^m(\mathbb{Q}_p^d))}=\operatorname{Tra}(bI_{\text{Sym}^m(\mathbb{Q}_p^d)})=\operatorname{Tra}(S_{\tau})=\sum_{j=1}^n\langle \tau_j^{\otimes m} , \tau_j^{\otimes m}  \rangle , \\
    	& b^2\operatorname{dim(\text{Sym}^m(\mathbb{Q}_p^d))}=\operatorname{Tra}(b^2I_{\text{Sym}^m(\mathbb{Q}_p^d)})=\operatorname{Tra}(S^2_{\tau})=\sum_{j=1}^n\sum_{k=1}^n\langle \tau_j^{\otimes m} , \tau^{\otimes m} _k \rangle\langle \tau^{\otimes m} _k, \tau^{\otimes m} _j \rangle.
    \end{align*}	
Therefore 

    \begin{align*}
    	&\left|\sum_{j=1}^n\langle \tau_j, \tau_j \rangle^m \right|^2=	\left|\sum_{j=1}^n\langle \tau_j^{\otimes m}, \tau_j^{\otimes m} \rangle \right|^2=|\operatorname{Tra}(S_{\tau})|^2=\left|b\operatorname{dim(\text{Sym}^m(\mathbb{Q}_p^d))}\right|^2\\
    	&=\left|\operatorname{dim(\text{Sym}^m(\mathbb{Q}_p^d))}\right|\left|b^2\operatorname{dim(\text{Sym}^m(\mathbb{Q}_p^d))}\right|\\
    	&=\left|\operatorname{dim(\text{Sym}^m(\mathbb{Q}_p^d))}\right|\left|\sum_{j=1}^n\sum_{k=1}^n\langle \tau_j^{\otimes m}, \tau_k^{\otimes m} \rangle\langle \tau_k^{\otimes m}, \tau_j^{\otimes m} \rangle\right|\\
    	&=\left|{d+m-1 \choose m}\right|\left|\sum_{j=1}^n\sum_{k=1}^n\langle \tau_j^{\otimes m}, \tau_k^{\otimes m} \rangle\langle \tau_k^{\otimes m}, \tau_j^{\otimes m} \rangle\right|\\
    	&=\left|{d+m-1 \choose m}\right|\left|\sum_{j=1}^n\sum_{k=1}^n\langle \tau_j, \tau_k \rangle^m\langle \tau_k, \tau_j \rangle^m\right|\\
    	&=\left|{d+m-1 \choose m}\right|\left| \sum_{l=1}^n\langle \tau_l,\tau_l \rangle^{2m}+\sum_{j,k=1, j\neq k}^n\langle \tau_j,\tau_k\rangle^m \langle \tau_k,\tau_j\rangle^m\right|\\
    	&\leq \left|{d+m-1 \choose m}\right| \max_{1\leq j,k \leq n, j \neq k}\left \{\left| \sum_{l=1}^n\langle \tau_l,\tau_l \rangle^{2m} \right|, |\langle \tau_j,\tau_k\rangle^m \langle \tau_k,\tau_j\rangle^m| \right\}\\
    	&= \left|{d+m-1 \choose m}\right| \max_{1\leq j,k \leq n, j \neq k}\left \{\left| \sum_{l=1}^n\langle \tau_l,\tau_l\rangle^{2m} \right|, |\langle \tau_j,\tau_k\rangle|^{2m}\right\}.
    \end{align*}
    Whenever $\langle \tau_j, \tau_j\rangle =1$ for all $1\leq j \leq n$, 
    \begin{align*}
    	|n|^2\leq  \left|{d+m-1 \choose m}\right| \max_{1\leq j,k \leq n, j \neq k}\{|n|, |\langle \tau_j, \tau_k\rangle|^{2m} \}.
    \end{align*}
  \end{proof}
\begin{remark}
	Conditions given in the Theorem \ref{WELCHNON2} says that the operator $S_\tau $ in the proof of Theorem \ref{WELCHNON2}	is diagonalizable. Thus Theorem \ref{WELCHNON2}  is restrictive as the   hypothesis is stronger than that   of Theorem 2.3 in \cite{MAHESHKRISHNA4}. However, note that the field $\mathbb{Q}_p$ does not satisfies the Equation (2) in \cite{MAHESHKRISHNA4} (see \cite{PEREZGARCIASCHIKHOF}) and hence neither the results in this paper can be derived from the results in \cite{MAHESHKRISHNA4} nor the results in \cite{MAHESHKRISHNA4} can be derived from the results in this paper. 
\end{remark}
\begin{remark}
	Theorems \ref{WELCHNON1}  and \ref{WELCHNON2} hold by replacing $\mathbb{Q}_p^d$ by a $d$-dimensional p-adic Hilbert space over any non-Archimedean (complete) valued field (such as $\mathbb{C}_p$).
\end{remark}

\section{p-adic Zauner Conjecture and open problems}
Using Theorem \ref{WELCHNON1} we ask the  following question.
\begin{question}\label{Q1}
	\textbf{Given a prime $p$, for which  $(d,n) 	\in \mathbb{N}\times \mathbb{N}$,  there exist vectors $\tau_1, \dots, \tau_n \in \mathbb{Q}_p^d$ satisfying the following.
		\begin{enumerate}[\upshape(i)]
			\item $\langle \tau_j, \tau_j \rangle =1$ for all $1\leq j \leq n$.
			\item  There exists $b \in \mathbb{Q}_p$ satisfying 
			\begin{align*}
				\sum_{j=1}^{n}\langle x, \tau_j\rangle \tau_j =bx, \quad \forall x \in \mathbb{Q}^d_p.
			\end{align*} 
			\item 
			\begin{align*}
				\max_{1\leq j,k \leq n, j \neq k}\{|n|, |\langle \tau_j, \tau_k\rangle|^2 \}= \frac{|n|^2}{|d|}.
			\end{align*}
	\end{enumerate}}
\end{question}
We can formulate a strong form of Question \ref{Q1} as follows. 
\begin{question}\label{Q2}
	\textbf{Given a prime $p$, for which  $(d,n) 	\in \mathbb{N}\times \mathbb{N}$,  there exist vectors $\tau_1, \dots, \tau_n \in \mathbb{Q}_p^d$ satisfying the following.
	\begin{enumerate}[\upshape(i)]
		\item $\langle \tau_j, \tau_j \rangle =1$ for all $1\leq j \leq n$.
		\item  There exists $b \in \mathbb{Q}_p$ satisfying 
		\begin{align*}
			\sum_{j=1}^{n}\langle x, \tau_j\rangle \tau_j =bx, \quad \forall x \in \mathbb{Q}^d_p.
		\end{align*} 
		\item 
		\begin{align*}
			\max_{1\leq j,k \leq n, j \neq k}\{|n|, |\langle \tau_j, \tau_k\rangle|^2 \}= \frac{|n|^2}{|d|}.
		\end{align*}
	\item $\|\tau_j\| =1$ for all $1\leq j \leq n$.
\end{enumerate}}	
\end{question}
Why Question \ref{Q2} is different than Question \ref{Q1}? Reason is that unlike non-Archimedean Hilbert spaces, in p-adic Hilbert spaces, norm is not defined as $\sqrt{|\langle \cdot, \cdot\rangle|}.$
A particular case of Question \ref{Q1} is the following p-adic  version of Zauner conjecture (see \cite{APPLEBY123, APPLEBY, ZAUNER, SCOTTGRASSL, FUCHSHOANGSTACEY, RENESBLUMEKOHOUTSCOTTCAVES, APPLEBYSYMM, BENGTSSON, APPLEBYFLAMMIAMCCONNELLYARD, KOPPCON, GOURKALEV, BENGTSSONZYCZKOWSKI, PAWELRUDNICKIZYCZKOWSKI, BENGTSSON123, MAGSINO, MAHESHKRISHNA} for Zauner conjecture in Hilbert spaces,   \cite{MAHESHKRISHNA2}  for Zauner conjecture in Hilbert C*-modules,  \cite{MAHESHKRISHNA3}  for Zauner conjecture in Banach spaces and \cite{MAHESHKRISHNA4} for Zauner conjecture in non-Archimedean Hilbert spaces).
\begin{conjecture}\label{NZ} \textbf{(p-adic Zauner Conjecture)
		Let $p$ be a prime.	For each $d\in \mathbb{N}$, there exist vectors $\tau_1, \dots, \tau_{d^2} \in \mathbb{Q}_p^d$ satisfying the following.
		\begin{enumerate}[\upshape(i)]
			\item $\langle \tau_j, \tau_j \rangle =1$ for all $1\leq j \leq d^2$.
			\item  There exists $b \in \mathbb{Q}_p$ satisfying 
			\begin{align*}
				\sum_{j=1}^{d^2}\langle x, \tau_j\rangle \tau_j =bx, \quad \forall x \in \mathbb{Q}^d_p.
			\end{align*} 
			\item 
			\begin{align*}
				|\langle \tau_j, \tau_k\rangle|^2 =|n|, \quad \forall 1\leq j, k \leq d^2, j \neq k.
			\end{align*}
	\end{enumerate}}
	\end{conjecture}
Question \ref{Q2} gives the following Zauner conjecture. 
\begin{conjecture}
\textbf{(p-adic Zauner Conjecture - strong form)
	Let $p$ be a prime.	For each $d\in \mathbb{N}$, there exist vectors $\tau_1, \dots, \tau_{d^2} \in \mathbb{Q}_p^d$ satisfying the following.
	\begin{enumerate}[\upshape(i)]
		\item $\langle \tau_j, \tau_j \rangle =1$ for all $1\leq j \leq d^2$.
		\item  There exists $b \in \mathbb{Q}_p$ satisfying 
		\begin{align*}
			\sum_{j=1}^{d^2}\langle x, \tau_j\rangle \tau_j =bx, \quad \forall x \in \mathbb{Q}^d_p.
		\end{align*} 
		\item 
		\begin{align*}
			|\langle \tau_j, \tau_k\rangle|^2 =|n|, \quad \forall 1\leq j, k \leq d^2, j \neq k.
		\end{align*}
	\item $\|\tau_j\| =1$ for all $1\leq j \leq d^2$.
\end{enumerate}}	
\end{conjecture}
We recall  the definition of Gerzon's bound which allows us to remember companions   to  Welch bounds in  Hilbert spaces. 
 \begin{definition}\cite{JASPERKINGMIXON}
	Given $d\in \mathbb{N}$, define \textbf{Gerzon's bound}
	\begin{align*}
		\mathcal{Z}(d, \mathbb{K})\coloneqq 
		\left\{ \begin{array}{cc} 
			d^2 & \quad \text{if} \quad \mathbb{K} =\mathbb{C}\\
			\frac{d(d+1)}{2} & \quad \text{if} \quad \mathbb{K} =\mathbb{R}.\\
		\end{array} \right.
	\end{align*}	
\end{definition}
\begin{theorem}\cite{JASPERKINGMIXON, XIACORRECTION, MUKKAVILLISABHAWALERKIPAAZHANG, SOLTANALIAN, BUKHCOX, CONWAYHARDINSLOANE, HAASHAMMENMIXON, RANKIN}  \label{LEVENSTEINBOUND}
	Define $\mathbb{K}=\mathbb{R}$ or $\mathbb{C}$ and    $m\coloneqq \operatorname{dim}_{\mathbb{R}}(\mathbb{K})/2$.	If	$\{\tau_j\}_{j=1}^n$  is any collection of  unit vectors in $\mathbb{K}^d$, then
	\begin{enumerate}[\upshape(i)]
		\item (\textbf{Bukh-Cox bound})
		\begin{align*}
			\max _{1\leq j,k \leq n, j\neq k}|\langle \tau_j, \tau_k\rangle |\geq \frac{\mathcal{Z}(n-d, \mathbb{K})}{n(1+m(n-d-1)\sqrt{m^{-1}+n-d})-\mathcal{Z}(n-d, \mathbb{K})}\quad \text{if} \quad n>d.
		\end{align*}
		\item (\textbf{Orthoplex/Rankin bound})	
		\begin{align*}
			\max _{1\leq j,k \leq n, j\neq k}|\langle \tau_j, \tau_k\rangle |\geq\frac{1}{\sqrt{d}} \quad \text{if} \quad n>\mathcal{Z}(d, \mathbb{K}).
		\end{align*}
		\item (\textbf{Levenstein bound})	
		\begin{align*}
			\max _{1\leq j,k \leq n, j\neq k}|\langle \tau_j, \tau_k\rangle |\geq \sqrt{\frac{n(m+1)-d(md+1)}{(n-d)(md+1)}} \quad \text{if} \quad n>\mathcal{Z}(d, \mathbb{K}).
		\end{align*}
		\item (\textbf{Exponential bound})
		\begin{align*}
			\max _{1\leq j,k \leq n, j\neq k}|\langle \tau_j, \tau_k\rangle |\geq 1-2n^{\frac{-1}{d-1}}.
		\end{align*}
	\end{enumerate}	
\end{theorem}
Theorem \ref{LEVENSTEINBOUND}  and Theorem \ref{WELCHNON1}  give  the following problem.
\begin{question}
	\textbf{Whether there is a   p-adic version of Theorem \ref{LEVENSTEINBOUND}? In particular, does there exists a   version of}
	\begin{enumerate}[\upshape(i)]
		\item \textbf{p-adic Bukh-Cox bound?}
		\item \textbf{p-adic Orthoplex/Rankin bound?}
		\item \textbf{p-adic Levenstein bound?}
		\item \textbf{p-adic Exponential bound?}
	\end{enumerate}		
\end{question}
We already wrote that  Welch bounds have applications in study of equiangular lines.  We wish to formulate equiangular line  problem for p-adic Hilbert spaces. For the study of equiangular lines in Hilbert spaces we refer \cite{LEMMENSSEIDEL, JIANGTIDORYAOZHAOZHAO, GREAVESKOOLENMUNEMASASZOLLOSI, BALLADRAXLERKEEVASHSUDAKOV, BUKH, DECAEN, GLAZYRINYU, BARGYU, JIANGPOLYANSKII, NEUMAIER, GREAVESSYATRIADIYATSYNA, GODSILROY, CALDERBANKCAMERONKANTORSEIDEL, OKUDAYU, YU}, quaternion Hilbert spaces we refer \cite{ETTAOUI}, octonion Hilbert spaces we refer \cite{COHNKUMARMINTON}, finite dimensional vector spaces over finite fields we refer \cite{GREAVESIVERSONJASPERMIXON1, GREAVESIVERSONJASPERMIXON2},  for Banach spaces we refer \cite{MAHESHKRISHNA3}  and for non-Archimedean Hilbert spaces we refer \cite{MAHESHKRISHNA4}.
\begin{question}\label{EQUI}
	\textbf{(p-adic Equiangular Line Problem)	Let $p$ be a prime. Given $a\in \mathbb{Q}_p$, $d \in \mathbb{N}$ and $\gamma>0$, what is the maximum $n =n(p, a,d, \gamma)\in \mathbb{N}$ such that there exist vectors $\tau_1, \dots, \tau_n \in \mathbb{Q}_p^d$ satisfying the following.
\begin{enumerate}[\upshape(i)]
	\item $\langle \tau_j, \tau_j \rangle =a$ for all $1\leq j \leq n$.
	\item $|\langle \tau_j, \tau_k \rangle|^2 =\gamma$ for all $1\leq j, k \leq n, j \neq k$.
\end{enumerate}
In particular, whether there is a p-adic  Gerzon bound?}
\end{question}
 Question \ref{EQUI} can be easily lifted  to  formulate question of   p-adic regular $s$-distance sets.

 \bibliographystyle{plain}
 \bibliography{reference.bib}

\end{document}